\documentclass[12pt]{amsart}
\usepackage{graphicx}
\usepackage{amsmath,amsthm, amsfonts,amssymb, stmaryrd,yfonts,pxfonts,pifont,eufrak, bbm}
\newtheorem{Cor}{Corollary}
 \newtheorem{Lemma}{Lemma}

 \newtheorem{Proposition}{Proposition}
 \theoremstyle{definition}
 
 \theoremstyle{remark}
 \newtheorem{Remark}[Lemma]{Remark}
 \numberwithin{equation}{subsection}

\begin{document}
\title[MATRIX CHARACTERIZATION OF MULTI-DIMENSIONAL SUBSHIFTS OF FINITE TYPE]{MATRIX CHARACTERIZATION OF MULTI-DIMENSIONAL SUBSHIFTS OF FINITE TYPE}%
\author{Puneet Sharma AND Dileep Kumar}
\address{Department of Mathematics, I.I.T. Jodhpur, Old Residency Road, Ratanada, Jodhpur-342011, INDIA}%
\email{puneet.iitd@yahoo.com}%




\begin{abstract}
Let $X\subset A^{Z^d}$ be a $2$-dimensional subshift of finite type.
We prove that any $2$-dimensional multidimensional subshift of
finite type can be characterized by a square matrix of infinite
dimension. We extend our result to a general $d$-dimensional case.
We prove that the multidimensional shift space is non-empty if and
only if the matrix obtained is of positive dimension. In the
process, we give an alternative view of the necessary and sufficient
conditions obtained for the non-emptiness of the multidimensional
shift space. We also give sufficient conditions for the shift space
$X$ to exhibit periodic points.
\end{abstract}
\maketitle

\section{INTRODUCTION}

The study of dynamical systems originated to facilitate the study of
natural processes and phenomenon. Many natural phenomenon can be
modeled as discrete dynamical systems and their long term behavior
can be approximated using the modeled system. However, investigating
a general discrete dynamical system is complex in nature and the
long term behavior of the system cannot always be determined
accurately. The uncertainty in predicting long term behavior
introduces dynamical complexity in the system which in turn results
in erroneous behavior of the modeled system. Thus, there is a need
to develop tools to facilitate the study of a general dynamical
system which are not erroneous and can model the physical system
with the sufficient accuracy. Symbolic dynamics is one of such tools
which are structurally simpler and can be used to model the physical
system with desired accuracy. In one of the early studies, Jacques
Hadamard used symbolic dynamics to study the geodesic flows on
surfaces of negative curvature\cite{had}. Claude Shennon applied
symbolic dynamics to the field of communication to develop the
mathematical theory of communication systems\cite{shanon}. Since
then the topic has found applications in areas like data storage,
data transmission and planetary motion to name a few. The area has
also found significant applications in various branches of science
and engineering\cite{bruce,lind1}. Its simpler structure and easy
computability can be used to investigate any general dynamical
system. Infact, it is known that every discrete dynamical system can
be embedded in a symbolic dynamical system with appropriate number
of symbols \cite{fu}. Thus, to investigate a general discrete
dynamical system, it is sufficient to study the shift spaces and its
subsystems.\\

Multidimensional shift spaces has been  a topic of interest to many
researchers. In one of the early works, Schmidt investigated
multidimensional subshifts of finite type over finite number of
symbols. He proved that for a multidimensional subshift, it is
algorithmically undecidable whether an allowed partial configuration
can be extended to a point in the multidimensional shift space. He
also observed that it is algorithmically undecidable to verify the
non-emptiness of a multidimensional shift defined by a set of finite
forbidden patterns. He also gave an example to show that a
multidimensional shift space may not contain any periodic
points\cite{sc}. These results unraveled the uncertainty associated
with a multidimensional shift space and attracted attention of
several researchers around the globe. As a result, several
researchers have explored the field and a lot of work has been
done\cite{quas,ban,sc,beal,boy1,hoch1,hoch4,sam}. In \cite{quas},
authors proved that multidimensional shifts of finite type with
positive topological entropy cannot be minimal. Infact, if $X$ is
subshift of finite type with positive topological entropy, then $X$
contains a subshift which is not of finite type, and hence contains
infinitely many subshifts of finite type. Quas and Trow in the same
paper proved that every shift space $X$ contains an entropy minimal
subshift $Y$, i.e., a subshift $Y$ of $X$ such that $h(Y)=h(X)$
\cite{quas}. While \cite{ban} investigated mixing properties of
multidimensional shift of finite type, \cite{beal} investigated
minimal forbidden patterns for multidimensional shift spaces. In
\cite{boy1}, authors exhibit mixing $\mathbb{Z}^d$ shifts of finite
type and sofic shifts with large entropy. However, they establish
that such systems exhibit poorly separated subsystems. They give
examples to show that while there exists $\mathbb{Z}^d$ mixing
systems such that no non-trivial full shift is a factor for such
systems, they provide examples of sofic systems where the only
minimal subsystem is a single point. In \cite{hoch1}, for
multidimensional shifts with $d \geq 2$, authors proved that a real
number $h\geq 0$ is the entropy of a $\mathbb{Z}^d$ shift of finite
type if and only if it is the infimum of a recursive sequence of
rational numbers. In \cite{hoch4}, Hochman improved the result and
showed that $h\geq 0$ is the entropy of a $\mathbb{Z}^d$ effective
dynamical system if and only if it is the lim inf of a recursive
sequence of rational numbers. The problem of determining which class
of shifts have a dense set of periodic points is still open. For
two-dimensional shifts, Lightwood proved that strongly irreducible
shifts of finite type have dense set of periodic points \cite{sam}.
However, the problem is still open for
shifts of dimension greater than two.\\


Let $A = \{a_i : i \in I\}$ be the a finite set and let $d$ be a
positive integer. Let the set $A$ be equipped with the discrete
metric and let $A^{\mathbb{Z}^d}$, the collection of all functions
$c : \mathbb{Z}^d \rightarrow A$ be equipped with the product
topology. The function $\mathcal{D} : A^{\mathbb{Z}^d} \times
A^{\mathbb{Z}^d} \rightarrow \mathbb{R}^+$ be defined as
$\mathcal{D} (x,y) = \frac{1}{n+1}$, where $n$ is the least natural
number such that $x \neq y$ in $R_n = [-n,n]^d$, is a metric on
$A^{\mathbb{Z}^d}$ and generates the product topology. For any $a\in
\mathbb{Z}^d$, the map $\sigma_a : A^{\mathbb{Z}^d} \rightarrow
A^{\mathbb{Z}^d}$ defined as $(\sigma_a (x))_k= x_{k+a}$ is a
$d$-dimensional shift and is a homeomorphism. For any $a,b\in
\mathbb{Z}^d$, $\sigma_a \circ \sigma_b = \sigma_b \circ \sigma_a$
and hence $\mathbb{Z}^d$ acts on $A^{\mathbb{Z}^d}$ through
commuting homeomorphisms. A set $X \subseteq A^{\mathbb{Z}^d}$ is
$\sigma_a$-invariant if $\sigma_a(X)\subseteq X$. Any set $X
\subseteq A^{\mathbb{Z}^d}$ is shift-invariant if it is invariant
under $\sigma_a$ for all $a \in {Z}^d$. A non-empty, closed shift
invariant subset of $A^{\mathbb{Z}^d}$ is called a shift space. If
$Y \subseteq X$ is a closed, nonempty shift invariant subset of $X$,
then $Y$ is called a subshift of $X$. For any nonempty $S\subset
\mathbb{Z}^d$, the projection map $\pi_S :A^{\mathbb{Z}^d}
\rightarrow A^S$ defined as $\pi_S= A^{\mathbb{Z}^d}|_S$ projects
the elements of $A^{\mathbb{Z}^d}$ to $S$. Any element in $A^S$ is
called a pattern over $S$. A pattern is said to be finite if it is
defined over a finite subset of $\mathbb{Z}^d$. A pattern $q$ over
$S$ is said to be extension of the pattern $p$ over $T$ if $T\subset
S$ and $q|_T=p$. The extension $q$ is said to be proper extension if
$S\cap Bd(T)=\phi$, where $Bd(T)$ denotes the boundary of $T$. For a
shift space $X$ and any set $S\subset \mathbb{Z}^d$, the set
$\mathcal{A}_S = \{ x\in A^S: x = \pi_S(y), \text{~~for some~~} y\in
X \}$ is called the set of allowed patterns(in $X$) over $S$. The
set $\mathcal{A}= \bigcup \limits_{S\subset \mathbb{Z}^d}
\mathcal{A}_S$ is called the set of allowed patterns or language for
the shift space $X$. Given a set $S\subset \mathbb{Z}^d$ and a set
of patterns $\mathcal{P}$ in $A^{S}$, the set $X= X(S,\mathcal{P}) =
\{x \in A^{\mathbb{Z}^d} : \pi_S\circ \sigma^n(x)\in \mathcal{P}
\text{~~for every n} \in \mathbb{Z}^d\}$ is a subshift generated by
the patterns $\mathcal{P}$. If the set $S$ is finite, the subshift
generated is a subshift of finite type. Refer \cite{quas,sc,beal}
for details.\\

Although the definition of a multidimensional shift is given in
terms of the allowed patterns, an equivalent definition can be given
in terms of forbidden patterns. Such a definition provides an
alternate view of the subshifts of finite type and in some cases can
be more beneficial for further investigations. For the sake of
completion, we provide the equivalent definition below.\\

For a shift space $X$ and any set $S\subset \mathbb{Z}^d$, the set
$\mathcal{F}_S = \{ x\in A^S: x\neq \pi_S(y) \text{~~for any~~} y\in
X \}$ is called the set of forbidden patterns(in $X$) over $S$. The
set $\mathcal{F}= \bigcup \limits_{S\subset \mathbb{Z}^d}
\mathcal{F}_S$ is called the set of forbidden patterns for the shift
space $X$. For a given set of patterns $\mathcal{F}$ (possibly over
different subsets of $\mathbb{Z}^d$), define,

\centerline{$X=\{x\in A^{\mathbb{Z}^d}: \text{any pattern in~~}
\mathcal{F} \text{~~does not appear in~~} x \}$}

It can be seen that the set defined above is a shift space. If
$\mathcal{F}$ contains finitely many patterns defined over finite
subsets of $\mathbb{Z}^d$, then the shift generated is a shift of
finite type. We denote the shift space generated by the set of
forbidden patterns $\mathcal{F}$ by $X_{\mathcal{F}}$. Two forbidden
sets $\mathcal{F}_1$ and $\mathcal{F}_2$ are said to be equivalent
if they generate the same shift space, i.e. $X_{\mathcal{F}_1}=
X_{\mathcal{F}_2}$. A forbidden set $\mathcal{F}$ of patterns is
called minimal for the shift space $X$ if $\mathcal{F}$ is the set
with least cardinality such that $X=X_{\mathcal{F}}$. It is worth
mentioning that a shift space $X$ is of finite type if its minimal
forbidden set is a finite set of finite patterns.\\





%

Let $M$ be a square matrix (possibly infinite) with indices
$\{\mathfrak{i} : \mathfrak{i}\in \mathfrak{I}\}$. We say that the
index $\mathfrak{i}$ is $u$-related to $\mathfrak{j}$ if
$M_{\mathfrak{ji}}=1$. Let the collection of indices $u$-related to
$\mathfrak{j}$ be denoted by $R^u_{\mathfrak{j}}$. We say that the
indices $\mathfrak{j}$ is $d$-related to $\mathfrak{i}$ if
$M_{\mathfrak{ji}}=1$. Let the collection of indices $d$-related to
$\mathfrak{i}$ be denoted by $R^d_{\mathfrak{i}}$. It may be noted
that $\mathfrak{i}$ is $u$-related to $\mathfrak{j}$ if and only if
$\mathfrak{j}$ is $d$-related to $\mathfrak{i}$. The set
$\mathfrak{I}$ is said to be complementary if for each
$\mathfrak{i}\in\mathfrak{I}$, there exists $\mathfrak{j},
\mathfrak{k}\in \mathfrak{I}$ such that $\mathfrak{j}$ is
$u$-related to $\mathfrak{i}$ and $\mathfrak{k}$ is $d$-related to
$\mathfrak{i}$.\\

In this paper we investigate some of the questions raised in
\cite{sc}. In the process we address the problem of non-emptiness
and existence of periodic points for a multidimensional shift of
finite type. We prove that the any $2$-dimensional shift of finite
type can be characterized by an infinite square matrix of
uncountable dimension. We extend our results to a general
$d$-dimensional case. We provide necessary and sufficient conditions
for a multidimensional subshift of finite type to be non-empty. In
the end, we also give sufficient condition for the subshift to
contain periodic
points. \\

\section{Main Results}

\begin{Proposition}
$X$ is a $d$-dimensional shift of finite type $\implies$  there
exists a set $\mathcal{C}$ of $d$-dimensional cubes such that
$X=X_{\mathcal{C}}$.
\end{Proposition}

\begin{proof}
Let $X$ be a shift of finite type and let $\mathcal{F}$ be the
minimal forbidden set of patterns for the shift space $X$. It may be
noted that $\mathcal{F}$ contains finitely many patterns defined
over finite subsets of $\mathbb{Z}^d$. For any pattern $p$ in
$\mathcal{F}$, let $l_p^i$ be the length of the pattern $p$ in the
i-th direction. Let $l_p = \max\{l_P^i : i=1,2,\ldots,d\}$ denote
the width of the pattern $p$ and let $l= \max\{ l_p : p \in
\mathcal{F}\}$. Let $\mathbb{C}_l$ be the collection of
$d$-dimensional cubes of length $l$ and let
$\mathbb{E}_{\mathcal{F}}$ denote the set of extensions of patterns
in $\mathcal{F}$.  Let $\mathcal{C}= \mathbb{C}_l \cap
\mathbb{E}_{\mathcal{F}}$.  It may be observed that if $p$ is a
pattern with width $l$, forbidding a pattern $p$ for $X$ is
equivalent to forbidding all extensions $q$ of $p$ in
$\mathcal{C}_l$. Thus, each pattern in the forbidden set of width
$l$ can be replaced by an equivalent forbidden set of cubes of
length $l$ and $\mathcal{C}$ is an equivalent forbidden set for the
shift space $X$. Consequently, $X=X_{\mathcal{C}}$ and the proof is
complete.
\end{proof}

\begin{Remark}
The above result proves that every $d$-dimensional shift of finite
type is generated by a set of cubes of fixed finite length. Such a
consideration leads to an equivalent forbidden set which in general
is not minimal. The above result constructs an equivalent forbidden
set by considering all the cubes which are extension of the set of
patterns in $\mathcal{F}$. However, the cardinality of the new set
can be reduced by considering only those cubes which are not proper
extensions of patterns in $\mathcal{F}$ (but are of same size $l$).
Such a construction reduces the cardinality of the forbidden set
considerably and hence reduces the complexity of the system. It may
be noted that the forbidden set obtained on reduction is still not
minimal. However, the $d$-dimensional cubes generating the elements
of $X$ are of same size and can be used advantageously for
constructing elements of $X$. We say that a shift of finite type $X$
is generated by cubes of length $l$ if there exists a set of cubes
$\mathcal{C}$ of length $l$ such that $X=X_{\mathcal{C}}$.
\end{Remark}

\begin{Proposition}
Every $2$-dimensional shift of finite type $X$ can be characterized
by an infinite square matrix.
\end{Proposition}

\begin{proof}
Let $X$ be a $2$-dimensional shift of finite type and let
$\mathcal{F}$ be the equivalent set of forbidden cubes (of fixed
length, say $l$) for the space $X$. Let $\mathcal{A}$ be the
generating set of cubes (of length $l$) for the space $X$. It may be
noted that as cubes of length $l$ form a generating set for the
shift space $X$, to verify whether any $x\in A^{\mathbb{Z}^d}$
belongs to $X$, it is sufficient to examine strips of height $l$ in $x$.\\

Let $\mathcal{A}^2= \{ \left(\begin{array}{l} A \\ B
\end{array}\right): A,B \in \mathcal{A},
\left(\begin{array}{l} A \\ B
\end{array}\right) \text{~~is allowed in~~} X \}$.

By construction, $\mathcal{A}^2$ is a finite set of $2n \times n$
allowed rectangles, say $\{a_1,a_2,\ldots,a_k\}$, generating the shift space $X$.\\

Define a $k\times k$ matrix $M$ as

\begin{center}
$M_{ij} = \left\{%
\begin{array}{ll}
    0, & \hbox{$(a_i a_j) \text{~~is forbidden in~~} X $;} \\
    1  & \hbox{$(a_i a_j) \text{~~is allowed in~~} X$;} \\
\end{array}%
\right.$
\end{center}

Then, the sequence space corresponding to the matrix $M$, $\Sigma_M
= \{ (x_n) : M_{x_i x_{i+1}}=1,~~ \forall i\}$ generates all allowed
infinite strips(of height $2l$) in $X$. It may be noted that any
element in $\Sigma_M$ is element of the form $\left(\begin{array}{l}
A
\\ B \end{array}\right)$, where $A$ and $B$ are allowed infinite strips of height $l$.


Generate an infinite matrix $\mathfrak{M}$, indexed by allowed
infinite strips of height $l$, using the following algorithm:

\begin{enumerate}
\item Pick any $\left(\begin{array}{l} A
\\ B \end{array}\right)\in \Sigma_M$ and index first two rows and columns of the
matrix by $A$ and $B$. Set $\mathfrak{m}_{BA}=1$.

\item For each $\left(\begin{array}{l} A
\\ B \end{array}\right)\in \Sigma_M$, if the rows and columns
indexed $A$ and $B$ exist, set $\mathfrak{m}_{BA}=1$. Else, label
next row and/or column as $A$ and/or $B$ (whichever required) and
set $\mathfrak{m}_{BA}=1$.

\item In the infinite matrix generated in step $2$, set $\mathfrak{m}_{BA}=0$,
if $\mathfrak{m}_{BA}$ has so far not been assigned a value.

\item In the infinite matrix obtained, if there exists an index $A$
such that the $A$-th row or column is zero, delete the $A$-th row
and column from the matrix generated.
\end{enumerate}

The above algorithm generates an infinite $0$-$1$ matrix where
$\mathfrak{m}_{BA}=1$ if and only if $\left(\begin{array}{l} A
\\ B \end{array}\right)$ is allowed in $X$, where $A$ and $B$ are allowed infinite strips (of height $l$) in
$X$. Let $\Sigma_{\mathfrak{M}}$ be the sequence space associated
with the matrix $\mathfrak{M}$. Consequently, any sequence in
$\Sigma_{\mathfrak{M}}$ gives a vertical arrangement of infinite
allowed strips (of height $l$) such that the arrangement is allowed
in $X$ and hence generates an element in $X$. Conversely, any
element in $X$ is a sequential (vertical) arrangement of infinite
strips of height $l$ and hence is generated by a sequence in
$\Sigma_{\mathfrak{M}}$. Consequently, $X= \Sigma_{\mathfrak{M}}$
and the proof is complete.
\end{proof}


\begin{Remark}
The above result characterizes elements of the shift space $X$ by a
infinite square matrix $\mathfrak{M}$. It may be noted that if
row/column for an index $A$ is zero, the algorithm deletes the row
and column with index $A$. Such a criteria reduces the size of the
matrix and will result in a matrix of dimension $0$, if the shift
space is empty. Further, the characterization of the space yields a
matrix of infinite (uncountable) dimension. Consequently, it is
undecidable whether a shift of finite type generated by set of cubes
$\mathcal{A}$ is non-empty. It may be noted that although the
algorithm does not guarantee a positive dimensional matrix, if the
shift space $X$ is non-empty the matrix generated is definitely of
positive dimension and characterizes the elements in $X$. Further,
as each row/column of the matrix generated has atleast one non-zero
entry, each block indexing the matrix can be extended to an element
of $X$. Consequently, any submatrix of the matrix $\mathfrak{M}$
cannot generate the shift space $X$. In light of the remark stated,
we get the following result.
\end{Remark}

\begin{Cor}
A $2$-dimensional shift of finite type is non-empty if and only if
the characterizing matrix $\mathfrak{M}$ is of positive dimension.
Further, any proper submatrix of the matrix $\mathfrak{M}$ generates
a proper subshift and hence the matrix $\mathfrak{M}$ is minimal.
\end{Cor}

\begin{Remark}
Although, in general it is undecidable whether a multidimensional
shift of finite type is non-empty, the non-emptiness problem can be
addressed using submatrices of the matrix $\mathfrak{M}$. In
particular, if there exists a submatrix $\mathfrak{N}$ of
$\mathfrak{M}$ (say generated after finite/countable steps of
algorithm) such that the space $\Sigma_{\mathfrak{N}}$ is non-empty,
the shift space $X$ is non-empty. As any non-empty shift space
characterized by a finite dimensional matrix contains periodic
points, such a verification (in finite time) cannot be conducted for
a shift space without periodic points. Further, such a verification
addresses only the non-emptiness problem and does not characterize
the elements in the shift space.
\end{Remark}

\begin{Remark}
For a shift space $X$, with generating set of cubes of height $l$,
let $\mathcal{L}$ denote set of all allowed infinite strips of
height $l$. Recalling the notions of $u$-related indices for a
square matrix $M$, for any two infinite strips $A,B$ of height $l$,
we say that $A$ is $u$-related ($d$-related) to $B$ if $A$ and $B$
are indices of $M$ such that $M_{BA}=1~~ (M_{AB}=1)$. Further,
generalizing the definition, a family of allowed infinite strips of
height $l$ is complementary if for each $A$ in $\mathcal{L}$ there
exists infinite strips $B,C\in\mathcal{L}$ such that $B$ is
$u$-related to $A$ and $C$ is $d$-related to $A$. Thus, the
algorithm generates $u$-related ($d$-related) infinite strips for
the shift space $X$ which in turn generates an arbitrary element of
$X$. As any element of the shift space is a sequential arrangement
of $u$-related ($d$-related) infinite strips, the characterization
of the elements of the space $X$ by a matrix $\mathfrak{M}$ is
equivalent to finding all the $u$-related ($d$-related) pairs of
infinite strips for the space $X$. As any infinite strip of height
$l$ (say $A$) can be extended to an element of $X$ only if there
exists infinite strips $B,C$ of height $l$ such that $B$ is
$u$-related to $A$ and $C$ is $d$-related to $A$, only members of
complementary family can form the building blocks for an element of
$X$. As a result, we get the following corollary.
\end{Remark}

\begin{Cor}
Let $X$ be a multidimensional shift space generated by cubes of
length $l$ and let $\mathfrak{B}$ be the infinite strips of height
$l$ allowed in $X$. Then, the shift space $X$ is non-empty if and
only if there exists non-empty set of indices
$\mathfrak{B}_0\subseteq \mathfrak{B}$ such that $\mathfrak{B}_0$ is
complementary.
\end{Cor}

\begin{Remark}
The above result provides an alternate view of the criteria
established for the non-emptiness of the space $X$. The result does
not require the matrix $\mathfrak{M}$ for establishing the
non-emptiness for the shift space. The set of indices during
construction of the matrix may be observed at each iteration and
existence of a complementary subfamily can be used to establish the
non-emptiness of the space $X$. However, as the algorithm does not
provide any
optimal technique for picking the block $\left(\begin{array}{l} A \\
B \end{array}\right)$ at each iteration,  such a consideration does
not reduce the time complexity and the problem of non-emptiness is
still undecidable. However, algorithms for optimal selection of the
infinite blocks $\left(\begin{array}{l} A \\ B \end{array}\right)$
may be proposed which in turn may reduce the time complexity of the
algorithm. As any multidimensional shift can be realized as an
extension of a $2$-dimensional shift, similar results are true for a
general $d$-dimensional shift of finite type. For the sake of
completion, we include the proof of main result below.
\end{Remark}

\begin{Proposition}
If $X$ is a $d$-dimensional shift of finite type, then the elements
of $X$ can be determined by an infinite square matrix.
\end{Proposition}

\begin{proof}
Let $X$ be a $d$-dimensional shift of finite type and let
$\mathcal{F}$ be the equivalent set of forbidden cubes (of fixed
length, say $l$) for the space $X$. Let $\mathcal{A}$ be the
generating set of cuboids of size $\underbrace{2l\times
2l\times\ldots 2l}_{d-1 times} \times l$  for the space $X$.

%

By construction, $\mathcal{A}$ is a finite set of allowed
rectangles, say $\{a_1,a_2,\ldots,a_k\}$. Define a $k\times k$
matrix $\mathfrak{M}^0$ as

\begin{center}
$\mathfrak{M}^0_{ij} = \left\{%
\begin{array}{ll}
    0, & \hbox{$(a_i a_j) \text{~~is forbidden in~~} X $;} \\
    1  & \hbox{$(a_i a_j) \text{~~is allowed in~~} X$;} \\
\end{array}%
\right.$
\end{center}

where $(a_i a_j)$ denotes adjacent placement of $a_j$ with $a_i$ in
the positive $d$-th direction.

Then, the sequence space corresponding to the matrix
$\mathfrak{M}^0$, $\Sigma_{\mathfrak{M}^0} = \{ (x_n) :
\mathfrak{M}^0_{x_i x_{i+1}}=1,~~ \forall i\}$ generates all allowed
one directional (in $d$-th direction) infinite strips in $X$.

It may be noted that any element in $\Sigma_{\mathfrak{M}^0}$ is
element of the form $\left(\begin{array}{l} A \\ B
\end{array}\right)_{0}$, where $A$ and $B$ are allowed infinite
strips (in direction $d$) of dimension $\underbrace{2l\times
2l\times\ldots 2l}_{d-2 times}\times l\times \infty$ and
$\left(\begin{array}{l} A \\ B
\end{array}\right)_{0}$ denotes adjacent placement of $B$ with $A$ in
the negative $d-1$-th direction.


Generate an infinite matrix $\mathfrak{M}^1$, indexed by allowed
infinite strips of dimension $\underbrace{2l\times 2l\times\ldots
2l}_{d-2 times}\times l\times \infty$ , using the following
algorithm:

\begin{enumerate}
\item Pick any $\left(\begin{array}{l} A
\\ B \end{array}\right)_0\in \Sigma_{\mathfrak{M}^0}$ and index first two rows and columns of the
matrix by $A$ and $B$. Set $\mathfrak{m}_{BA}=1$.

\item For each $\left(\begin{array}{l} A
\\ B \end{array}\right)_0\in \Sigma_{\mathfrak{M}^0}$, if the rows and columns
indexed $A$ and $B$ exist, set $\mathfrak{m}_{BA}=1$. Else, label
next row and/or column as $A$ and/or $B$ (whichever required) and
set $\mathfrak{m}_{BA}=1$.

\item In the infinite matrix generated in step $2$, set $\mathfrak{m}_{BA}=0$,
if $\mathfrak{m}_{BA}$ has so far not been assigned a value.

\item In the infinite matrix obtained, if there exists an index $A$
such that the $A$-th row or column is zero, delete the $A$-th row
and column from the matrix.
\end{enumerate}

The above algorithm generates an infinite $0$-$1$ matrix where
$\mathfrak{m}_{BA}=1$ if and only if $\left(\begin{array}{l} A
\\ B \end{array}\right)_0$ is allowed in $X$, where $A$ and $B$ are of
dimension $\underbrace{2l\times 2l\times\ldots 2l}_{d-2 times}\times
l\times \infty$. Let $\Sigma_{\mathfrak{M}^1}$ denote the sequence
space corresponding to the matrix generated above. It can be seen
that the space $\Sigma_{\mathfrak{M}^1}$ precisely is the collection
of allowed bi-infinite strips (in direction $d$ and $d-1$). Further,
as any element in $\Sigma_{\mathfrak{M}^1}$ is of the form
$\left(\begin{array}{l} A \\ B
\end{array}\right)_{1}$, where $A$ and $B$ are allowed infinite
strips (in direction $d$ and $d-1$) of dimension
$\underbrace{2l\times 2l\times\ldots 2l}_{d-3 times}\times l\times
\infty\times \infty$ and $\left(\begin{array}{l} A \\ B
\end{array}\right)_{1}$ denotes adjacent placement of $B$ with $A$ in
the negative $d-2$-th direction, a repeated application of the
algorithm generates a matrix $\mathfrak{M}^2$ which extends the
infinite patterns in $\Sigma_{\mathfrak{M}^1}$ along the direction
$d-3$ to generate the space $\Sigma_{\mathfrak{M}^2}$. Consequently,
repeated application of the above algorithm extends the allowed
patterns infinitely in all the $d$ directions (one direction at each
step) to obtain a point in $X$. Further, as any point in $X$ can be
visualized as such an extension of allowed cubes in the $d$
directions, the matrix obtained (at the final step) characterizes
the elements of the space $X$.


\end{proof}

\begin{Remark}
The above result characterizes the multidimensional shift space by a
matrix $\mathfrak{M}$. The result is a repeated application of the
$2$-dimensional case, extending the allowed block in each of the $d$
directions. In the process, at each step $i$ we obtain an infinite
matrix characterizing the extension of an allowed block in the
$i$-th direction. Although the rows and columns of the
characterizing matrix $\mathfrak{M}$ are indexed by infinite blocks
allowed in $X$, their existence/verfication is beyond any ambiguity
as they are algorithmically generated. It may be noted that
extension in any of the directions (at step $i$) does not guarantee
an extension to the element of $X$. In particular, a block
extendable in a direction $i$ (or in a few directions
$i_1,i_2,\ldots,i_r$) need not necessarily extend to an element in
$X$. In particular if the shift space is empty, although we may
obtain matrices of positive dimension in initial few steps, the
final matrix obtained characterizing the elements of $X$ is
$0$-dimensional. Consequently, once again, the shift space is
non-empty if and only if the matrix generated (at the final step) is
of positive dimension. As the algorithm is an extension of the
algorithm for the $2$-dimensional case, results similar to the
$2$-dimensional case also hold good for any general dimension $d\geq
3$. For the sake of completion, we mention the generalizations
below.
\end{Remark}

\begin{Cor}
A multidimensional shift of finite type is non-empty if and only if
the characterizing matrix $\mathfrak{M}$ is of positive dimension.
Further, any proper submatrix of the matrix $\mathfrak{M}$ generates
a proper subshift and hence the matrix $\mathfrak{M}$ is minimal.
\end{Cor}

\begin{Cor}
Let $X$ be a multidimensional shift space and let $\mathfrak{B}$ be
the infinite strips of height $l$ allowed in $X$. Then, the shift
space $X$ is non-empty if and only if there exists
$\mathfrak{B}_0\subseteq \mathfrak{B}$ such that $\mathfrak{B}_0$ is
complementary.
\end{Cor}

We now give some results relating the matrix $\mathfrak"{M}$ and the
dynamical behavior of the shift space $X$.

\begin{Proposition}
Let $X$ be a multidimensional shift space and let $\mathfrak{B}$ be
the infinite strips of height $l$ allowed in $X$. If there exists a
finite complementary set $\mathfrak{B}_0\subset \mathfrak{B}$, then
the set of periodic points is non-empty.
\end{Proposition}

\begin{proof}
Let $\mathfrak{B}$ be the infinite strips of height $l$ allowed in
$X$ and let $\mathfrak{B}_0\subset \mathfrak{B}$ be a finite
complementary set. By definition, elements of $\mathfrak{B}_0$ form
indices (not all) for the matrix $\mathfrak{M}$. Let $\mathfrak{N}$
be the submatrix of $\mathfrak{M}$ indexed by elements of
$\mathfrak{B}_0$. As the set $\mathfrak{B}_0$ is complementary, the
shift generated by $\mathfrak{B}_0$ (say $\Sigma_{\mathfrak{B}_0}$)
is non-empty. Further, as shift defined by a finite dimensional
matrix contains periodic points, there exists periodic points for
$\Sigma_{\mathfrak{B}_0}$ (and hence for the shift space $X$).
\end{proof}

\section{Conclusion}

In this paper, we investigate the non-emptiness problem for a
multidimensional shift space of finite type. In the process, we
prove that any multidimensional shift of finite type can be
characterized by an infinite square matrix of uncountable dimension.
We prove that the multidimensional shift space is non-empty if and
only if the matrix $\mathfrak{M}$ is of positive dimension. We also
prove that any submatrix of the matrix obtained generates a proper
subshift of $X$ and hence the matrix $\mathfrak{M}$ minimally
generates the elements of $X$. We further observe that non-emptiness
of such a shift may be examined using complementary set of indices.
However, construction of such a family of indices is non-trivial and
may not be possible in finite time. Consequently, the non-emptiness
problem for such a space is undecidable. We also provide a
sufficient condition for a multidimensional shift of finite type to
exhibit periodic points.\\

\bibliography{xbib}

\end{document}